\documentclass[11pt]{amsart}
\usepackage {amsmath, amscd}
\usepackage {amssymb}
\usepackage {bm}

\def\<{\langle}
\def\>{\rangle}

\def\C{{\mathbb{C}}}
\def\D{\mathbb{D}}

\def\T{\mathbb{T}}

\def\HH{{\mathcal H}}

\def\KK{{\mathcal K}}

\def\MM{{\mathcal M}}

\def\YY{{\mathcal Y}}

\def\bbD{\mathbb{D}}

\renewcommand{\leq}{\leqslant}
\renewcommand{\geq}{\geqslant}
\renewcommand{\subset}{\subseteq}

\renewcommand{\Re}{\mathop{\rm Re}}

 \usepackage{palatino}

 \numberwithin{equation}{section}
\theoremstyle{plain}

\newtheorem{corollary}[equation]{Corollary}
\newtheorem{theorem}[equation]{Theorem}
\newtheorem{lemma}[equation]{Lemma}
\theoremstyle{definition}

\newtheorem{example}[equation]{Example}

\newtheorem{remark}[equation]{Remark}
 
\usepackage{enumitem}
\setlist[enumerate]{leftmargin=*}
\setlist[itemize]{leftmargin=*}

\setlist[enumerate,1]{label=(\alph*),font=\upshape}

\setlist[enumerate,2]{label=(\roman*),font=\upshape}

\title[Littlewood in de Branges--Rovnyak spaces]{Littlewood subordination for de Branges--Rovnyak spaces}

\author[Fricain]{Emmanuel Fricain}
 \address{Laboratoire Paul Painlev\'e, Universit\'e de Lille, 59 655 Villeneuve d'Ascq C\'edex }
 \email{emmanuel.fricain@univ-lille.fr}

\author[Hartmann]{Andreas Hartmann}
\address{Univ. Bordeaux, CNRS, Bordeaux INP, IMB, UMR 5251, F-33400, Talence, France}
\email{Andreas.Hartmann@math.u-bordeaux.fr}

\author[Ross]{William T. Ross}
	\address{Department of Mathematics and Statistics, University of Richmond, Richmond, VA 23173, USA}
	\email{wross@richmond.edu}
	
		\author[Timotin]{Dan Timotin}
	\address{Simion Stoilow Institute of Mathematics of the Romanian Academy, Calea Grivi\c tei 21, Bucharest 010702, Romania}
	\email{Dan.Timotin@imar.ro}

\begin{document}

\begin{abstract}
In this paper, we study composition operators that act between different de Branges–Rovnyak spaces. Our main results are suggested by a paper of Mashreghi and Shabankhah concerning composition operators between model spaces. We also answer several related open questions posed by Dellepiane and Seco. To prove these results, we apply reproducing kernel Hilbert space methods, Sarason’s approach to composition operators as integral operators, and Aleksandrov–Clark measures.
\end{abstract}

\maketitle	

\section{Introduction}

Littlewood's subordination principle provides a natural starting point for studying composition operators on various Banach spaces of analytic functions on the open unit disk $\D = \{z \in \C: |z| < 1\}$. The original Littlewood theorem from \cite{MR1574067} (see also \cite[p.~10]{Duren}) states that if $\phi$ is an analytic self-map of $\D$ and a function $f$ belongs to the Hardy space $H^2$ \cite{Duren},  then the composition $f \circ \phi$ also belongs to $H^2$. This implies that  
$$C_{\phi} f := f \circ \phi$$  defines a bounded linear operator  on $H^2$. Moreover, with the additional assumption that $\phi(0) = 0$,  the composition operator $C_{\phi}$ on $H^2$ becomes a contraction.  Littlewood’s theorem has motivated extensive research on the boundedness and compactness of composition operators between various Banach spaces of analytic functions on $\D$ \cite{CMC, Queff, JShap}.

In this paper, we focus on  Littlewood subordination problems for the de Branges--Rovnyak spaces $\mathcal{H}(b)$. Here $b \in H^{\infty}_{1}$, the closed unit ball in the algebra  $H^{\infty}$ of bounded analytic functions on $\D$, and $\HH(b)$ is  the Hilbert space  of analytic functions on $\D$ with the positive semi-definite reproducing kernel 
\begin{equation}\label{KernelHb}
k^{b}_{\lambda}(z) := \frac{1 -  \overline{b(\lambda)} b(z)} {1 - \overline{\lambda} z}, \quad \lambda, z \in \D.
\end{equation}
It is known that $\HH(b)$ is contractively contained in $H^2$. When $b \equiv 0$, then $\HH(b) = H^2$ with equal norms. More generally, if the closure of $b(\D)$ is contained in $\D$, then $\HH(b) = H^2$ with  equivalent norms. 
When $b$ is an inner function, $\HH(b)$ becomes the  model space $\KK_{b} := H^2 \ominus b H^2$. Of course, when $b$ is a unimodular constant function, we have $\HH(b)  \equiv \{0\}$ and so we assume from now on that this is never the case. See \cite{Hb1-book, Hb2-book, MR1289670} for detailed treatments of $\HH(b)$ spaces and \cite{MR3526203} for more on model spaces.

The Littlewood subordination principle, along with the contractive inclusion of $\HH(b)$ in $H^2$,  implies that for any analytic self-map $\phi$ of $\D$, the composition operator $C_\phi$ maps
$\mathcal{H}(b)$ boundedly into $H^2$. Inspired by a discussion from \cite{MR3411049}, concerning the case when $b$ is inner, the authors in \cite{MR3915413} prove some necessary or sufficient conditions for the compactness of $C_\phi:\mathcal H(b)\to H^2$. On the other hand, the question of whether $C_\phi$ defines a bounded operator  from $\mathcal H(b)$ to itself is studied in \cite{MR4653340} when $b \in H^{\infty}_{1}$ is a rational function (and not inner) and for weighted composition operators in \cite{FricainOstermann}.  It turns out  that $C_{\phi}$ maps $\mathcal{H}(b)$ to itself only under severe
restrictions on the symbol $\phi$. 

In this paper, we extend the discussion to general $b, \phi \in H^{\infty}_{1}$, focusing on finding a function $\eta\in H_1^\infty$, depending on  $b$ and $\phi$, for which $C_\phi$  maps $\mathcal{H}(b)$ boundedly to
$\mathcal{H}(\eta)$. Of course, one can take $\eta \equiv 0$ in which case $\HH(\eta) = H^2$ (as discussed above). However, our primary interest lies in identifying smaller target spaces $ \HH(\eta) $ tailored to particular $ b $ and $ \phi $.

The inspiration for our main result comes from a paper of Mashreghi and Shabankhah \cite{MR3176147} who studied the above  composition operator problem when $b$ and $\phi$ are inner functions, where their proofs depend heavily on this assumption. A recent paper of Dellepiane and Seco \cite{DELLEPIANE2027130974} studied this problem when $\phi(z)=z^n$ and $b \in H^\infty_1$ satisfy $\log(1-|b|)\in L^1(m)$ ($m$ is normalized Lebesgue measure on the unit circle $\T$). In our paper, not only do we cover general $b$ and $\phi$, but our approach is new compared to previous papers and makes use of reproducing kernel Hilbert space techniques, an idea of Sarason that defines composition operators on the space of complex measures on the unit circle $\T$, and the theory of Aleksandrov--Clark measures. 

To contextualize our results, we recall the following theorem, which covers composition operators between two model spaces. This situation arises when $b$ and $\phi$ are inner functions.  

\begin{theorem}[Mashreghi--Shabankhah]\label{Javad}
	For inner functions $\phi$ and $b$, define the inner function  $ \eta$ as follows:
	\begin{enumerate}
		\item If $ b(0)=0 $, then $$\eta(z)=	z\frac{(b\circ \phi)(z)}{\phi(z)}.$$
		
		\item If  $ b(0)\not=0$ and $ \phi(0)=0 $, then $ \eta(z)=	(b\circ\phi) (z) $.
		
		\item  If  $ b(0)\not=0$ and $ \phi(0)\not=0 $, then $ \eta(z)=	z(b\circ\phi) (z) $.
		
	\end{enumerate}
	Then   the composition operator $C_\phi$ maps $\mathcal{K}_b$ boundedly to $\mathcal{K}_{\eta}$.
	Moreover, if \begin{equation}\label{bkshift}
(S^{*} f)(z) = \frac{f(z) - f(0)}{z}
\end{equation}
denotes the backward shift on $H^2$, 
then    $\mathcal{K}_{\eta}$ is the smallest $S^*$-invariant subspace of $H^2$
	that contains $C_{\phi}(\mathcal{K}_{b})$. 
\end{theorem}

Our main result is the following analog of Theorem \ref{Javad} for composition operators between two de Branges--Rovnyak spaces. 

\begin{theorem}\label{MainT}
For $b$ and $\phi$ in $H^{\infty}_{1}$ we have the following.
\begin{enumerate}
\item The operator 
$C_{\phi}$ maps $\HH(b)$ boundedly to $ \HH(z (b \circ \phi))$.
\item If $b(0) = 0$, then $C_{\phi}$ maps $\HH(b)$ boundedly to $\HH(z \frac{b \circ \phi}{\phi})$.
\item If $\HH(\phi)$ contains the constants, then $C_{\phi}$ maps $\HH(b)$ boundedly to $\HH(b \circ \phi)$. 
\end{enumerate}
Moreover,  if $\phi(0) = 0$, then, in the three cases above, $C_{\phi}$ acts contractively on $\HH(b)$.
\end{theorem}

More tangible conditions for $\HH(\phi)$ to contain the constants are given in Lemma \ref{le:1 in H9phi} below. 

When $\phi$ is inner and $\phi(0)=0$, we can say quite a bit more about the range of $C_{\phi}$ on $\HH(b)$.  

\begin{theorem}\label{onepointsix}
If $b \in H^{\infty}_{1}$, $\phi$ is inner with  $\phi(0) = 0$, and $\mathcal{Y}_{\phi} := C_{\phi}(H^2)$, then we have the following.
\begin{enumerate}
    \item $C_{\phi}$  is an isometry from $\mathcal{H}(b)$ to $\mathcal{H}(b \circ \phi)$ with  $$C_\phi(\HH(b))=\HH(b\circ\phi)\cap\YY_\phi.$$ 
    \item $\HH(b\circ\phi)$ has the orthogonal decomposition 	
    \begin{equation}\label{lsdfjjjJJ}
        \HH(b\circ\phi)= C_\phi(\HH(b))\oplus_{	\HH(b\circ\phi)} 	\left(\HH(b\circ\phi)\cap\YY_\phi^\perp\right),
    \end{equation}
	where $\mathcal{Y}_{\phi}^{\perp} = H^2 \ominus \mathcal{Y}_{\phi}$. 
\end{enumerate}
\end{theorem}

	The previous theorem answers several open questions posed by Dellepiane and Seco \cite{DELLEPIANE2027130974}. Specifically, Theorem~\ref{onepointsix} (a) extends \cite[Theorem 3.9]{DELLEPIANE2027130974}, where they prove that  $C_\phi:\HH(b)\to \HH(b\circ\phi)$ is an isometry when $\phi(z)=z^n$ and $\log(1 - |b|) \in L^1(m)$. Problem 3.10 in \cite{DELLEPIANE2027130974} asks about the relationship between $\HH(b\circ\phi)$ and $\HH(b)$ and Theorem~\ref{onepointsix} (b) gives a partial answer when $ \phi $ is inner and vanishes at $0$. 

The alert reader probably noticed the absence of a discussion of minimality in Theorem \ref{MainT} as there was in Theorem \ref{Javad}. We will address this issue in the last section of our paper.  

This paper is organized as follows. Section \ref{RKfacts} reviews standard facts about reproducing kernels, which are then used in Section  \ref{se:simple results} to prove  Theorem \ref{th:main} (stated as Theorem \ref{MainT} (c) above). Section  \ref{Clark} discusses Sarason's extension of composition operators to measures on the unit circle via their Poisson integrals. This discussion sets  up the proofs of Theorem \ref{th:the general result} and Theorem \ref{th:b(0)0} in Section \ref{GeneralC}, which establish 
(a) and (b) of Theorem \ref{MainT}. In Section \ref{Structure},
we establish the orthogonal decomposition of $\HH(b  \circ \phi)$ given in \eqref{lsdfjjjJJ}. Finally, Section \ref{FRem} addresses the complexities of minimality in Theorem \ref{MainT} (as in Theorem \ref{Javad}) and presents further examples.   

\section{Reproducing kernel facts}\label{RKfacts}

A key element in proving Theorem \ref{MainT} is to analyze the reproducing kernel from \eqref{KernelHb}. This analysis relies on several general properties of reproducing kernel Hilbert spaces. For a detailed treatment of these concepts, see \cite{MR3526117}. For a set $X$, we say that a function $K: X \times X \to \C$ is positive semi-definite, written $K\gg 0$ if
$$\sum_{1 \leq j, k \leq n} K(x_j, x_k) c_j \overline{c_k} \geq 0$$
for each $n \geq 1$ and any $x_1, \ldots, x_n \in X$ and $c_1, \ldots, c_n \in \C$. 
We will usually abbreviate this condition as ``$K$ is positive''. 
To each positive $K$, called a kernel, there is an associated reproducing kernel Hilbert space $\mathcal{H}_K$ whose elements are functions $f: X \to \C$ satisfying 
$$f(x) = \langle f, K(\cdot, x)\rangle_{\mathcal{H}_{K}}, \quad x \in X.$$
The previous identity assumes that  
\begin{equation}\label{containskernels}
K(\cdot, x) \in \HH_{K} \; \; \mbox{for all $x \in X$}.
\end{equation}
The following are some useful facts about kernels.

\begin{lemma}\label{le:basic}
	\begin{enumerate}
		\item  Sums and products of positive kernels are positive.
		
		\item If $K$ is positive and $\phi:X\to X$, then $K(\phi(\cdot),\phi(\cdot))$ is positive.
		\item 
		For two positive kernels $ K_1$ and $K_2 $, 
		the following are equivalent: 
		\begin{enumerate}
		\item The composition operator $C_\phi$ is bounded from $\mathcal{H}_{K_2}$ to $\mathcal{H}_{K_1}$; 
		\item There exists a constant $M>0$ such that $$M^2K_1-K_2(\phi(\cdot), \phi(\cdot))\gg0.$$
		\end{enumerate}
		In this case, $\|C_\phi\|_{\HH_{K_2} \to \HH_{K_1}}$ is the smallest such $M$.
		\item For each $b\in H^\infty_1$, the kernel
		\[
		K(z, \lambda) = \frac{1- \overline{b(\lambda)}b(z)}{1-\overline{\lambda} z}, \quad z, \lambda \in \D,
		\]
		is positive. This is actually the reproducing kernel for the space $\HH(b)$ from \eqref{KernelHb}.
        \item If $f: X \to \C$, then $f \in \HH_{K}$ if and only if there is an $M > 0$ such that $M^2 K(x, y) - f(x) \overline{f(y)}\gg0$. Moreover,  $\|f\|_{\HH_{K}}$ is the smallest such $M$. 
        \item For any $f: X \to \C$, the kernel $K(x, y) = f(x) \overline{f(y)}$ is positive. 
	\end{enumerate}
\end{lemma}

Setting $\phi(z) = z$, condition (c) yields the following criterion for inclusion between two reproducing kernel Hilbert spaces. 

\begin{corollary}\label{co:inclusion between RKHS}
	For two kernels $ K_1, K_2 $, the following are equivalent. 
    \begin{enumerate}
        \item $\mathcal{H}_{K_2}\subset \mathcal{H}_{K_1}$;
        \item There exists a constant  $M>0$ such that $M^2K_1-K_2 \gg0$.
    \end{enumerate}
\end{corollary}

An application of the closed graph theorem implies  that any inclusion between two reproducing kernel Hilbert spaces is automatically bounded.
 
 \section{Direct reproducing kernel arguments}\label{se:simple results}
 
This section presents results obtained through reproducing kernel arguments, beginning with a direct consequence of Lemma \ref{le:basic} (e).

\begin{lemma}\label{le:1 in H9phi} For $ \phi\in H^\infty_1 $ the following are equivalent. 
\begin{enumerate}
    \item $ \HH(\phi) $ contains the constants;
    \item There is a constant $ M > 0 $ such that 
	\begin{equation}\label{eq:1 in Hphi}
	 \frac{M^2(1-\overline{\phi(\lambda)}\phi(z))}{1- \overline{\lambda} z}-1\gg0.
	\end{equation}
\end{enumerate}
	In this case, the minimal value of $ M $ is $ \| 1 \|_{\HH(\phi)} $.
\end{lemma}


The following lemma discusses the criterion as to when the constants belong to  $\HH(\phi)$. From   \cite[Thm.~6.7]{Hb1-book}, it is known that  $\log(1 - |\phi|) \in L^1(m)$ if and only if $\phi$ is a nonextreme point of $H^{\infty}_{1}$. In the above, recall that $m$ is normalized Lebesgue measure on $\T$. We will often use these terms ``extreme'' and ``nonextreme'' below. 

\begin{lemma}\label{Hbcopnstants}
For $\phi \in H^{\infty}_{1}$ the following are equivalent. 
\begin{enumerate}
\item $\HH(\phi)$ contains the constants; 
\item Either $\phi(0) = 0$ or $\phi$ is nonextreme.
\end{enumerate}
\end{lemma}

\begin{proof}
Observe that $k_0^{\phi}=1-\overline{\phi(0)}\phi$ and the reproducing kernels always belong to $\HH(\phi)$ (recall \eqref{containskernels}). Thus, $1\in \HH(\phi)$ if and only if either $\phi(0) = 0$ or $\overline{\phi(0)}\phi$, and hence $\phi$, belongs to $\HH(\phi)$.  However, $\phi\in \HH(\phi)$ if and only if $\phi$ is nonextreme (see \cite[Cor.~23.24]{Hb2-book} and \cite[Cor.~25.8]{Hb2-book}). 
\end{proof}

\begin{remark}\label{mate}
It is known  that $\|1\|_{\HH(\phi)} = 1$ when $\phi(0) = 0$ or 
$$\|1\|_{\HH(\phi)}^2 = 1 + \frac{|\phi(0)|^2}{|\psi(0)|^2},$$
where $\phi$ is nonextreme and 
$\psi \in H^{\infty}_{1}$ is the unique, up to a rotational constant, function that satisfies $|\phi|^2 + |\psi|^2 = 1$ almost everywhere on $\T$ \cite[Thm.~23.24]{Hb2-book}. The function $\psi$ is often called the Pythagorean mate of $\phi$. For particularly simple nonextreme $\phi$, one can compute its Pythagorean mate via the Fej\'{e}r--Riesz theorem \cite[Sec.~6]{MR3503356}. 
\end{remark}

We have the following result.

\begin{theorem}\label{th:main}
	
	Suppose $b, \phi\in H^\infty_1$ and $ \HH(\phi) $ contains the constants. Then 	
	$$C_\phi(\HH(b))\subset \HH(b\circ \phi)$$ and $\|C_{\phi}\|_{\HH(b) \to \HH(b \circ \phi)} \leq  \| 1 \|_{\HH(\phi)} $.
	
\end{theorem}
\begin{proof}
	By Lemma~\ref{le:basic} (c), it suffices to prove the existence of a constant $M > 0$ such that 
	\[
M^2	\frac{1-b(\phi(z))\overline{b(\phi(\lambda ))}}{1-z \overline{\lambda}}-\frac{1-b(\phi(z))\overline{b(\phi(\lambda))}}{1-\phi(z)\overline{\phi(\lambda)}}\gg0.
	\]
	To see this, observe that 
	\[
	\begin{split}
		&M^2	\frac{1-b(\phi(z))\overline{b(\phi(\lambda))}}{1-z \overline{\lambda}}-\frac{1-b(\phi(z))\overline{b(\phi(\lambda))}}{1-\phi(z)\overline{\phi(\lambda)}}\\
		&=\frac{1-b(\phi(z))\overline{b(\phi(\lambda)}}{1-\phi(z)\overline{\phi(\lambda)}}\left( \frac{M^2(1-\phi(z)\overline{\phi(\lambda)})}{1-z \overline{\lambda}}-1\right).
	\end{split}
	\]
		The first factor above is positive by (b) and (d) of 
 Lemma~\ref{le:basic} (since it is the composition of the kernel for $ \HH(b) $ with an analytic self-map $ \phi $). Setting $M = \|1\|_{\HH(\phi)}$ in Lemma \ref{le:1 in H9phi}, the second factor above is positive.
\end{proof}

Combining Lemma \ref{Hbcopnstants} and Theorem \ref{th:main} yields the following corollary.

\begin{corollary}\label{co:who contains the constants}
Suppose $b, \phi\in H^\infty_1$.

\begin{enumerate}
	\item If $ \phi(0)=0 $, then 	$C_\phi(\HH(b))\subset \HH(b\circ \phi)$ and $\|C_{\phi}\|_{\HH(b) \to \HH(b \circ \phi)} = 1.$
	
	\item If $\phi$ is nonextreme, then 	$C_\phi(\HH(b))\subset \HH(b\circ \phi)$.
	
\end{enumerate}
\end{corollary}

\begin{proof}
It suffices to show that if $\phi(0) = 0$, then  $$\|C_{\phi}\|_{\HH(b) \to \HH(b \circ \phi)} = 1.$$ Lemma \ref{Hbcopnstants} and Theorem \ref{th:main} say that 
$$\|C_{\phi}\|_{\HH(b) \to \HH(b \circ \phi)} \leq 1.$$
Since $\phi(0) = 0$, we have 
$C_{\phi}(k_{0}^{b}) = k_{0}^{b \circ \phi}$
and so 
\begin{align*}
    1 - |b(0)|^2 & = 1 - |b(\phi(0))|^2\\
    & = \|k_{0}^{b \circ \phi}\|^2_{\HH(b \circ \phi)}\\
    & \leq \|C_{\phi}\|^{2}_{\HH(b) \to \HH(b \circ \phi)} \|k_{0}^{b}\|^2_{\HH(b)}\\
    & = \|C_{\phi}\|^{2}_{\HH(b) \to \HH(b \circ \phi)} (1 - |b(0)|^2).
\end{align*}
Thus, $$\|C_{\phi}\|_{\HH(b) \to \HH(b \circ \phi)} \geq 1$$ and equality follows. 
\end{proof}

Here is our next case.

\begin{theorem}\label{th:b(0)0, phi(0)0}
	If $b, \phi \in H^{\infty}_{1}$ with $b(0)=\phi(0)=0$, then 
	\begin{equation*}\label{eq:b(0)=0old}
			C_\phi(\HH(b))\subset \HH\Big(z \frac{b\circ\phi}{\phi}\Big)
	\end{equation*}
	and the inclusion is contractive.
\end{theorem}

\begin{proof}
	The condition $b(0) = 0$ implies that $b(z):=zb_1(z)$, with $b_1 \in H^{\infty}_{1}$, and so 
	$$\frac{b(\phi(z))}{\phi(z)}=b_1(\phi(z)).$$  By Lemma~\ref{le:basic}, it suffices to establish the positivity of the kernel
\[
K(z,w):=\frac{1-z \overline{w} b_1(\phi(z))\overline{b_1(\phi(w))}}{1-z \overline{w}}-\frac{1-\phi(z)\overline{\phi(w)} b_1(\phi(z))\overline{b_1(\phi(w))}}{1-\phi(z)\overline{\phi(w)}}.
\]	
A computation verifies  that $K(z, w)$ is equal to 
\begin{align*}
	&\frac{z\overline{w}-\phi(z)\overline{\phi(w)}-z\overline{w} b_1(\phi(z))\overline{b_1(\phi(w))}+\phi(z))\overline{\phi(w)}b_1(\phi(z))\overline{b_1(\phi(w))}}{(1-z\overline{w})(1-\phi(z)\overline{\phi(w)})}\\
	&= \frac{(z\overline{w}-\phi(z)\overline{\phi(w)})(1-b_1(\phi(z))\overline{b_1(\phi(w))})}{(1-z\overline{w})(1-\phi(z)\overline{\phi(w)})}\\
    & = \Big(\frac{1-b_1(\phi(z))\overline{b_1(\phi(w))}}{1-\phi(z)\overline{\phi(w)}}\Big) \Big(\frac{z\overline{w}-\phi(z)\overline{\phi(w)}}{1-z\overline{w}}\Big).
\end{align*}
The first factor
	is positive by Lemma~\ref{le:basic} (b). The  condition $\phi(0)=0$ implies  that $\phi(z)=z\phi_1(z)$, with $\phi_1 \in H^{\infty}_{1}$. The remaining factor above becomes
	\[
	\frac{z\overline{w}-\phi(z)\overline{\phi(w)}}{1-z\overline{w}}=z\bar w \frac{1-\phi_1(z)\overline{\phi_1(w)}}{1-z\overline{w}}
	\] 
	which is a positive kernel. This can be seen by applying Lemma \ref{le:basic} (d), the fact that $z \bar w$ defines a positive kernel (Lemma \ref{le:basic} (f)) and Lemma \ref{le:basic} (a). A final application of Lemma \ref{le:basic} (a) yields the positivity of $K(z, w)$.
\end{proof}

As shown below, Theorem~\ref{th:b(0)0, phi(0)0} is a special  case of Theorem~\ref{th:b(0)0}. We include this simpler proof because it relies only on the basic properties of reproducing kernels.

\section{The extension of $C_\phi$ to measures on $\T$}\label{Clark}

The preceding sections address certain choices of $b$ and $\phi$ by algebraic manipulation of the reproducing kernels. The remainder of this paper completes our investigation of the general case. To establish these unrestricted results, this section makes a brief detour into the theory of Aleksandrov–Clark measures and an idea of Sarason.

For a complex Borel measure $\mu$ on the unit circle $\T$, define its Poisson integral $P_{\mu}$ by 
\[
P_{\mu}(z)=\int_{\T}\frac{1-|z|^2}{|\zeta-z|^2}\,d\mu(\zeta), \quad z \in \D,
\] 
and observe  that $P_{\mu}$ is harmonic on $\D$. 
For an analytic self-map $\phi$ of $\D$, the composition operator $C_\phi f = f \circ \phi$ is naturally defined on spaces of analytic functions on $\D$. Following the seminal work of Sarason \cite{sarason1990composition}, the composition operator $C_{\phi }$ can be viewed as acting on the space $\mathcal{M}$ of complex Borel measures on $\T$. The definition is quite natural: for each positive  $\mu\in\MM$, Herglotz's theorem \cite[Ch.~9]{CMR} says that the positive harmonic function $P_{\mu} \circ \phi$ is the Poisson integral of another positive measure $A_\phi(\mu) \in  \MM$. Properties of $\mu$ often transfer to properties of $A_\phi(\mu)$. 
Some good references and discussions for the main result we state in this section are \cite{MR2876277}, \cite[Thm.~5.1]{MR2394657}, \cite{MR994787}. Since we need a particular form of the result, and for the sake of completeness, we give an independent exposition.

For $\psi\in H^\infty_1$,
$$w \mapsto \Re\left(\frac{1+\psi(w)}{1-\psi(w)}\right)$$ defines a positive harmonic function on $\D$ and by Herglotz's theorem, this is the Poisson integral of a positive $\sigma \in \mathcal{M}$. In other words, 
\begin{equation}\label{eq:AC1}
	\Re\left(\frac{1+\psi(w)}{1-\psi(w)}\right)=\int_{\T}\frac{1-|w|^2}{|\zeta-w|^2}\, d\sigma(\zeta), \quad w \in \D.
\end{equation} 
The measure $\sigma$ is the  Aleksandrov--Clark measure for $\psi$.

Let $\nu \in \mathcal{M}$ be the Aleksandrov-Clark measure for $\psi \circ \phi$ in that
\begin{equation}\label{eq:AC2}
	\Re\left(\frac{1+(\psi \circ \phi)(w)}{1-(\psi \circ \phi)(w)}\right)=\int_{\T}\frac{1-|w|^2}{|\zeta-w|^2}\, d\nu(\zeta), \quad w \in \D.
\end{equation}
This $\nu$ is the measure $A_\phi(\mu)$ discussed earlier. 
The following is important for our results below. 

\begin{theorem}\label{th:Aleksandrov} Fix $\phi, \psi \in H^{\infty}_{1}$. For each $f\in L^p(\sigma)$, $1\leq p\leq \infty$, there exists a unique function $A_{\phi, \psi} f\in L^p(\nu)$ such that 
	\begin{equation}\label{eq:AC3}
		P_{fd\sigma}(\phi(z))=\int_{\T} \frac{1-|z|^2}{|\zeta-z|^2}(A_{\phi, \psi} f)(\zeta)\,d\nu(\zeta), \quad z \in \D.
	\end{equation}
	Moreover, the map $A_{\phi, \psi}: L^p(\sigma)\to L^p(\nu)$ is linear and bounded. In particular, 
	\begin{equation}\label{eq:L2norm}
		\|A_{\phi, \psi}\|_{L^2(\sigma) \to L^2(\nu)} \leq 4\,\sqrt{\frac{1+|\phi(0)|}{1-|\phi(0)|}}.
	\end{equation}
\end{theorem}

\begin{proof}
	If we set $w=\phi(z)$ in the  formula~\eqref{eq:AC1}, the left hand side is precisely the left hand side of~\eqref{eq:AC2} evaluated at $z$. Therefore, we obtain 
	\[
	P_\sigma \circ \phi = P_{\nu}.
	\]
	
Let $f\in L^\infty(\sigma)$ such that $0\leq f(\zeta)\leq 1$ $\sigma$-a.e. Since the Poisson kernel 
$$\frac{1-|z|^2}{|\zeta-z|^2}, \quad \zeta \in \T, z \in \D,$$ is strictly positive, we have 
\[
0\leq P_{fd\sigma}(w)\leq P_\sigma(w), \quad w \in \D,
\]	
whence
\[
0\leq P_{fd\sigma}(\phi(z))\leq P_\sigma(\phi(z))=P_\nu(z), \quad z \in \D.
\]
	Since $P_{fd\sigma}$ is harmonic (and positive) and $\phi$ is analytic, $P_{fd\sigma}\circ\phi$ is a positive harmonic function in $\bbD$, and so, again by Herglotz's theorem, this function  is the Poisson integral of a positive Borel measure $\mu_f$. In other words, 
	\begin{equation}\label{eq:compo for measures}
		P_{fd\sigma} \circ \phi =P_{\mu_f}.
	\end{equation}
	 The inequality above becomes
	\[
	0\leq P_{\mu_f}(z)\leq P_\nu(z), \quad z \in \D.
	\]
	This means that the Poisson integral of $\nu-\mu_f$
is positive, whence 	$\nu-\mu_f$ itself is a positive measure, i.e., $\mu_f\leq \nu$.
Therefore, $\mu_f$ is absolutely continuous with respect to $\nu$ and has Radon--Nikodym derivative $\widetilde f\in L^1(\nu)$ which satisfies $0\leq \widetilde f(\eta)\leq 1$ for $\nu$-almost $\eta\in \T$.

Define  $$A_{\phi, \psi}:L^\infty(\sigma)\to L^\infty(\nu),$$  first for positive functions by $A_{\phi, \psi} f=\widetilde f$ and then, by linearity, for all $f\in L^\infty(\sigma)$. Then~\eqref{eq:compo for measures} becomes
\begin{equation}\label{eq:AC6.5}
	P_{fd\sigma} \circ \phi =P_{(A_{\phi, \psi} f)d\nu}.
\end{equation}
Since computing its norm only uses positive functions, the above computation shows that 
\begin{equation}\label{0009iIU}
\|A_{\phi, \psi}\|_{L^{\infty}(\sigma) \to L^{\infty}(\nu)} \leq 4,
\end{equation}
which  shows the boundedness of the operator $A_{\phi, \psi}$ between the two $L^{\infty}$ spaces. 

Our next step is to prove the boundedness of $A_{\phi, \psi}$ between the two $L^1$ spaces. First, let $f\in L^\infty(\sigma)$, $f\geq 0$. We know that $\widetilde f= A_{\phi, \psi} f\geq 0$, so, via the mean value property for harmonic functions, its $L^1(\nu)$ norm is the value at $0$ of the Poisson integral  
of $\mu_f$. Therefore, using~\eqref{eq:compo for measures},
\[
\|A_{\phi, \psi} f\|_{L^1(\nu)} =P_{\mu_f}(0)=P_f(\phi(0))
= \int_{\T} \frac{1-|\phi(0)|^2}{|\zeta-\phi(0)|^2} f(\zeta) \,d\sigma(\zeta).
\]
An easy geometric argument yields 
\[
\frac{1-|\phi(0)|^2}{|\zeta-\phi(0)|^2}\leq \frac{1+|\phi(0)|}{1-|\phi(0)|},
\]
whence 
\[
\|A_{\phi, \psi} f\|_{L^1(\nu)}\leq \frac{1+|\phi(0)|}{1-|\phi(0)|}\int_{\T}f(\zeta)\,d\sigma(\zeta) =\frac{1+|\phi(0)|}{1-|\phi(0)|} \|f\|_{L^1(\sigma)}.
\]
From here we can extend the definition of $A_{\phi, \psi}$ by continuity and linearity to all of  $L^1(\sigma)$ and obtain the norm estimate 
\begin{equation}\label{77shdfsdf}
\|A_{\phi, \psi} f\|_{L^1(\nu)}\leq 4 \frac{1+|\phi(0)|}{1-|\phi(0)|} \|f\|_{L^1(\sigma)}.
\end{equation}
Moreover, the action of $A_{\phi, \psi}$ on a general $f \in L^1(\sigma)$ is described by 
$$P_{fd\sigma} \circ \phi=P_{(A_{\phi, \psi} f) d\nu}.$$  Indeed, if $f_n \in L^{\infty}(\sigma)$ and $f_n\to f$ in $L^1(\sigma)$, then their Poisson integrals converge pointwise to $P_{(A_{\phi, \psi} f) d\nu}$ in $\D$, and so, passing to the limit yields
\[
P_{fd\sigma}(\phi(z))=\int_{\T}\frac{1-|z|^2}{|\eta-z|^2}(A_{\phi, \psi} f)(\eta)\, d\nu(\eta).
\]

The Riesz--Thorin interpolation theorem now implies that the operator $A_{\phi, \psi}:L^p(\sigma)\to L^p(\nu)$ is bounded for all $1 \leq p \leq \infty$. Using \eqref{0009iIU} and \eqref{77shdfsdf} and the operator norm estimate provided by the Riesz--Thorin interpolation theorem, $\|A_{\phi, \psi}\|_{L^2(\sigma) \to L^2(\mu)}$ satisfies \eqref{eq:L2norm}. 
\end{proof}

The previous discussion has some important implications for the Cauchy transform $\mathcal{K}_{\sigma} f$ of a function $ f\in L^2(\sigma) $, which is defined by the integral formula 
\[
(\KK_\sigma f)(z):=\int_{\T}\frac{f(\zeta)}{1-z\overline{\zeta}}\, d\sigma(\zeta), \quad z \in \D.
\]
Note that $\mathcal{K}_{\sigma} f$ is an analytic function on $\D$.
We first assume that $ f $ is real-valued. Since $\zeta \in \T$, we have 
\begin{equation}\label{eq:AC decomposition}
\Re \frac{1}{1-z\overline{\zeta}}=\frac{1}{2}\left( \frac{1-|z|^2}{|\zeta-z|^2}+1 \right), \quad z \in \D,
\end{equation}
and so 
\[
\Re \left[(\KK_\sigma f)(z)\right] =\tfrac{1}{2}P_{f}(z)+
\tfrac{1}{2}\int_{\T}f\,d\sigma.
\]
Observe that the second term on the right hand side of the previous line is a real constant.
Composing the last equality with $ \phi $, we obtain the following: 
\[
\Re \left[(\KK_\sigma f)(\phi(z))\right]= \tfrac{1}{2}P_{f}(\phi(z))+
\tfrac{1}{2}\int_{\T} f\,d\sigma.
\]
Using~\eqref{eq:AC6.5}, this last equality can be written as
\begin{equation}\label{eq:AC4}
\Re \left[(\KK_\sigma f)(\phi(z))\right]= \tfrac{1}{2}P_{(A_{\phi, \psi} f) d\nu}(z)+
\tfrac{1}{2}\int_{\T} f\,d\sigma.
\end{equation}

On the other hand, we note that 
\[
K_\nu (A_{\phi, \psi} f)(z)=\int_{\T}\frac{(A_{\phi, \psi} f)(\eta)}{1-z\overline{\eta}}\,d\nu(\eta).
\]
Another use of \eqref{eq:AC decomposition} gives us the following: 
\begin{equation}\label{eq:AC5}
\Re \left[K_\nu (A_{\phi, \psi} f)(z)\right]=\tfrac{1}{2}P_{(A_{\phi, \psi} f)d\nu}(z)+\tfrac{1}{2}\int_{\T}A_{\phi, \psi} f\, d\nu.
\end{equation}

Comparing~\eqref{eq:AC4} and~\eqref{eq:AC5}, it follows that the real part of the analytic function $(\KK_\sigma f)(\phi(z))-K_{\nu}(A_{\phi, \psi} f)(z) $ is constant, and so the function itself is constant in that 
\begin{equation*}
	(\KK_\sigma f)(\phi(z))-K_\nu (A_{\phi, \psi} f)(z)=\xi\in\C.
\end{equation*}
Since the constant $\xi$ depends linearly on $f$, we can write it as a linear functional $\xi(f)$. 
Evaluating the above when $z = 0$, we obtain 
\[
\begin{split}
\xi(f)&= 	(\KK_\sigma f)(\phi(0))-K_\nu (A_{\phi, \psi} f)(0)\\
&=\int_{\T}\frac{f(\zeta)}{1-\phi(0)\overline{\zeta}}\, d\sigma(\zeta)-\int_{\T}(A_{\phi, \psi} f)(\eta)\,d\nu(\eta).
\end{split}
\]
However, 
\begin{align*}
|\xi(f)| & \leq \Big|\Big\langle f, \frac{1}{1-\overline{\phi(0)} \zeta}\Big\rangle_{L^2(\sigma)}\Big|+|\<A_{\phi, \psi} f, 1\>_{L^2(\nu)}|\\
&\leq \text{\rm const}\| f \|_{L^2(\sigma)},
\end{align*}
where we have used that $ A_{\phi, \psi}  $ is a bounded operator from $ L^2(\sigma) $ to $ L^2(\nu) $ (Theorem \ref{th:Aleksandrov}). Therefore, by the Riesz representation theorem, there is a  unique $G \in L^2(\sigma)$ such that $\xi(f) = \langle f, G\rangle_{L^2(\sigma)}$. We record this result in the following corollary for later use.

\begin{corollary}\label{co:Aleksandrov Cauchy}
Suppose $\phi, \psi\in H^\infty_1$, $\sigma, \nu$ are the Aleksandrov--Clark measures for $\psi$ and $\psi\circ\phi$ respectively, and $A_{\phi,\psi}:L^2(\sigma)\to L^2(\nu)$ is the operator defined in Theorem~\ref{th:Aleksandrov}. Then there exists  a unique $G\in L^2(\sigma)$, such that for all $f \in L^2(\sigma)$ and all $z \in \D$ we have
	\begin{equation}\label{eq:AC6}
		(\KK_\sigma f)(\phi(z))-K_\nu (A_{\phi, \psi} f)(z) =\< f,G \>_{L^2(\sigma)}.
	\end{equation}
\end{corollary}

\section{General composition results}\label{GeneralC}

We start with a simple lemma. Recall the backward shift operator $S^{*}$ on $H^2$ from \eqref{bkshift}.

\begin{lemma}\label{le:action of S*}
 For  $\psi\in H^\infty_1$ and $g \in H^2$, the following are equivalent. 
    \begin{enumerate}
        \item $g\in\HH(z\psi)$;
        \item  $S^*g\in\HH(\psi)$.
    \end{enumerate} 
\end{lemma}

\begin{proof}
	(a) $\implies$ (b): From \cite[Thm.~18.8]{Hb2-book} we know that $$\HH(z\psi)=\C\oplus z\HH(\psi).$$ Thus, if $g\in\HH(z\psi)$, we have $g=a+zg_1$ with $ g_1\in\HH(\psi)$, and $$S^*g=S^*(zg_1)=g_1.$$ 
    
    (b) $\implies$ (a): Suppose $g_1:=S^*g\in\HH(\psi)$. Then $$Sg_1=SS^*g=g-g(0),$$ or equivalently, $$g=g(0)+Sg_1\in \C \oplus z\HH(\psi)=\HH(z\psi),$$
    which completes the proof. 
\end{proof}

The following lemma originates from the work of Clark, who proved it for inner functions. The general case is due to Ball and Lubin and a discussion of this is found in \cite[Ch.~II, IV]{MR1289670}. For a positive Borel measure $\mu$ on $\T$, let $H^2(\mu)$ denote the closure of $\C[z]$ in $L^2(\mu)$. 

\begin{lemma}\label{le:cauchy integral rep}
	For a function $\psi\in H^\infty_1$, let $\sigma$ be its Aleksandrov--Clark measure from \eqref{eq:AC1}. Then the linear transformation 
	\[
	V_\psi f(z):=(1-\psi(z))\int_{\T} \frac{f(\zeta)}{1- \overline{\zeta} z}\,d\sigma(\zeta), \quad z \in \D,
	\]
	is a partial isometry from $L^2(\sigma)$ onto $\HH(\psi)$, with $\ker(V_\psi)=(H^2(\sigma))^\perp$.
\end{lemma}

The next lemma is the central result of this section.

\begin{lemma}\label{le:my pet result}
	Suppose $\psi ,\phi\in H^\infty_1$. Then
	\[
	C_\phi(z\HH(\psi ))\subset \HH(z(\psi \circ\phi)).
	\]
\end{lemma}

\begin{proof}	Take $f_1=zh$, with $h\in\HH(\psi )$. We need to prove that $$\phi (h\circ\phi)\in  \HH(z(\psi \circ\phi)).$$ By Lemma~\ref{le:action of S*}, the above is equivalent to
	\begin{equation}\label{eq:to prove2}
		S^*(\phi (h\circ\phi))\in \HH(\psi \circ\phi).
	\end{equation}
	
	To prove this inclusion, we start by using Lemma~\ref{le:cauchy integral rep} and write
	\[
	h(z)=(1-\psi (z))\int_{\T} \frac{f(\zeta)}{1-z\overline{\zeta}}\,d\sigma(\zeta)
	\] 
	for some $f\in L^2(\sigma)$.
	Then
	\[
	h(\phi(z))=(1-\psi  (\phi(z)))\int_{\T} \frac{f(\zeta)}{1-\phi(z)\overline{\zeta}}\,d\sigma(\zeta).
	\]
	Multiplying the previous equation by $\phi(z)$ and using the identity $$\frac{\phi(z)}{1-\phi(z)\overline{\zeta}}=\frac{\zeta}{1-\phi(z)\overline{\zeta}}-\zeta, \quad z \in \D, \zeta \in \T,$$ we obtain
	\[
	\begin{split}
		\phi(z)h(\phi(z))& = (1-\psi  (\phi(z)))\left(\int_{\T} \frac{\zeta f(\zeta)}{1-\phi(z)\overline{\zeta}}\,d\sigma(\zeta)
		- \int_{\T} \zeta f(\zeta)\,d\sigma(\zeta)\right) \\
		&=(1-\psi  (\phi(z)))\left(\KK_\sigma(\zeta f)(\phi(z))-\int_{\T}\zeta f(\zeta) \, d\sigma (\zeta)\right).
	\end{split}
	\]
	Corollary~\ref{co:Aleksandrov Cauchy} gives us 
	\[
	\KK_\sigma(\zeta f)(\phi(z))=K_\nu (A_{\phi, \psi} (\zeta f))(z)+\< \zeta f, G \>_{L^2(\sigma)},
	\]
	and therefore, writing $$B=\< \zeta f, G-1 \>_{L^2(\sigma)},$$ we see that 
	$$
		\phi(z)h(\phi(z))=(1-\psi  (\phi(z)))K_\nu (A_{\phi, \psi}(\zeta f))(z)+B(1-\psi  (\phi(z))).
	$$
	Since $\nu$ is the Aleksandrov--Clark measure for $\psi \circ \phi$, we can apply Lemma~\ref{le:cauchy integral rep}, applied to $\psi  \circ\phi$, to see that the first term belongs to  $\HH(\psi  \circ\phi)$, as will $S^{*}$ applied to this term (de Branges--Rovnyak spaces are $S^{*}$-invariant \cite[Thm.~18.13]{Hb2-book}). Moreover, 
	\[
	BS^*(1-\psi  \circ\phi)=BS^*(\psi  \circ\phi).
	\]
A general property of de Branges--Rovnyak spaces says $$ S^*(\psi  \circ\phi)\in \HH(\psi\circ\phi) $$ \cite[Thm.~18.18]{Hb2-book}, which completes the proof. 	
	\end{proof}

Lemma \ref{le:my pet result} yields the following two important theorems.

\begin{theorem}\label{th:the general result}
	If $b,\phi\in H^\infty_1$, then
	\[
	C_\phi(\HH(b))\subset \HH(z(b\circ\phi)).
	\]
    Moreover, if $\phi(0)=0$, then $C_\phi$ acts contractively from $\HH(b)$ into $\HH(z(b\circ \phi)$.
\end{theorem}

\begin{proof}
	We have $\HH(b)\subset\HH(zb)=\C\oplus z\HH(b)$ \cite[Thm.~18.8]{Hb2-book}.  Obviously $$C_\phi(\C)=\C\subset \HH(z(b\circ\phi)).$$ Setting $\psi = b$ in Lemma~\ref{le:my pet result} implies that $$C_\phi(z\HH(b))\subset \HH(z(b\circ\phi)),$$ which completes the proof of the fact that $C_\phi(\HH(b))\subset \HH(z(b\circ\phi))$. Moreover, since $\HH(b\circ\phi)$ is contractively contained in $\HH(z(b\circ\phi))$ (see Lemma 7.3), if $\phi(0)=0$, then according to Corollary 3.6 (a), for every $f\in\HH(b)$ we have 
\[ 
\|f\circ\phi\|_{\HH(z(b\circ\phi))}\leq \|f\circ\phi\|_{\HH(b\circ\phi)}\leq \|f\|_{\HH(b)},
\]
which completes the proof. 
\end{proof}

Here is an example that shows why the factor $z$ in $z (b \circ \phi)$ is essential without further assumptions on $b$ or $\phi$. 
If $\lambda \in \D \setminus \{0\}$ and 
$$\phi(z) = b(z) = \frac{\lambda - z}{1 - \overline{\lambda} z},$$
then $\phi$ and $b$ are inner functions, and, via \cite[p.~116]{MR3526203},
$$\HH(b) = \mathcal{K}_b = \C \frac{1}{1 - \overline{\lambda} z}.$$
Moreover, 
$$C_{\phi}(\HH(b)) = \C (1 - \bar \lambda z) \subset  \mathcal{K}_{z^2} = \operatorname{span}\{1, z\}.$$
However, $(b \circ \phi)(z) = z$ and so, again, via \cite[p.~116]{MR3526203},
$$\HH(b \circ \phi) = \mathcal{K}_z = \C 1,$$
while 
$$\HH(z (b \circ \phi)) = \HH(z^2) = \mathcal{K}_{z^2} = \operatorname{span}\{1, z\}.$$
 The paper \cite{MR3176147} contains further examples of this phenomenon. 

 We finish this section with the proof of Theorem \ref{MainT} (b).

\begin{theorem}\label{th:b(0)0}
	If $b(0)=0$, then 
	\begin{equation}\label{eq:b(0)=0}
		C_\phi(\HH(b))\subset \HH \Big(z\frac{b\circ\phi}{\phi}\Big).
	\end{equation}

\end{theorem}

\begin{proof}
	Write $b(z)=zb_1(z)$, where $b_1 \in H^{\infty}_{1}$. Then~\eqref{eq:b(0)=0} is equivalent to
\begin{equation}\label{eq:b(0)=0bis}
	C_\phi(\HH(b))\subset \HH\left(z (b_1\circ\phi)\right).
\end{equation}
Again, from \cite[Thm.~18.8]{Hb2-book}, we have the orthogonal decompositions
\begin{equation}\label{eq:orth}
	\HH(b)=\C\oplus z\HH(b_1), \quad \HH\left(z (b_1\circ\phi)\right)=\C\oplus z\HH(b_1\circ\phi).
\end{equation}

Any $f\in\HH(b)$ can be decomposed as $f(z)=f(0)\oplus zh(z)$ with $h\in\HH(b_1)$, and so
\[
C_\phi(f)=f(0)+\phi(z)h(\phi(z)).
\]
The first term above belongs to $\HH(z(b_1\circ\phi))$ by~\eqref{eq:orth}, so it suffices to prove that $\phi (h\circ\phi)\in  \HH(z(b_1\circ\phi))$. But last fact follows from Lemma~\ref{le:my pet result} by setting $\psi =b_1$.	
\end{proof}


Theorems~\ref{th:the general result} and~\ref{th:b(0)0}, together with Theorem \ref{th:main}, complete the proof of the announced Theorem \ref{MainT}.

\begin{remark}
The assumptions in Theorems \ref{th:main} and \ref{th:b(0)0}  are not exclusive. Indeed, suppose that $ b(0)=0 $ and $\phi$ is nonextreme, in which case $ \HH(\phi) $ contains the constant functions (Lemma \ref{Hbcopnstants}). Then, according to Theorem \ref{th:main}, $C_{\phi}$ maps $\HH(b)$ to $\HH(b \circ \phi)$. Moreover, according to 
Theorem \ref{th:b(0)0}, $C_{\phi}$ maps $\HH(b)$ to $\HH(z \frac{b \circ \phi}{\phi}).$ 
In particular, if $\phi(0) \not = 0$, the two spaces $\HH(b \circ \phi)$ and $\HH(z \frac{b \circ \phi}{\phi})$ are a priori not comparable.
\end{remark}


\section{ The structure of $C_\phi(\HH(b))$}\label{Structure}

In this section, we focus our attention on the case where $\phi $ is an  inner function and $ \phi(0)=0 $. This restriction  allows for a more detailed characterization of how $C_\phi$ acts on $\HH(b)$. Under these hypotheses, $ C_\phi $ is an isometry on $ H^2 $  \cite{MR223914}, 
enabling us to define the closed subspace 
$$\YY_\phi := C_{\phi}(H^2).$$ 
Although the following lemma is likely known, we provide a proof here for completeness as we were unable to find it in the literature. From our earlier discussion in the proof of Lemma \ref{Hbcopnstants},  $T_{w}$ denotes the Toeplitz operator on $H^2$ with the symbol $w \in L^{\infty}(\frac{d \theta}{2 \pi})$. 

\begin{lemma}\label{le:mistery identity}
	Let $\phi$ be inner with $\phi(0)=0$ and $\alpha\in H^\infty$. Then
	\begin{equation}\label{eq:mystery}
			T_{\overline{\alpha\circ\phi}}C_\phi=C_\phi T_{\overline\alpha}
	\end{equation}
Consequently, $\YY_\phi^\perp = H^2 \ominus \mathcal{Y}_{\phi}$ is invariant with respect to multiplication by the function $\alpha\circ\phi$.
\end{lemma}

\begin{proof}
	The formula in \eqref{eq:mystery} can be quickly verified when $\alpha$ is a polynomial. From here, the formula can be extended to functions $\alpha$ that are analytic in a neighborhood of $\overline{\D}$. Finally, if $\alpha_r(z):=\alpha(rz)$, where $z \in \D$ and  $0 < r < 1$, is a dilation of $\alpha$, then $T_{\overline{\alpha_r}}\to T_{\overline\alpha} $ in the weak operator topology, and so we can extend~\eqref{eq:mystery} to any $\alpha\in H^\infty$.
	It follows from~\eqref{eq:mystery}  that $\YY_\phi$ is an invariant subspace for $T_{\overline{\alpha\circ\phi}} $, whence $\YY_\phi^\perp$ is an invariant subspace for its adjoint $T_{\alpha\circ\phi}$ (which is multiplication by $\alpha\circ\phi$).
	\end{proof}

For typesetting purposes, we set 
$$\|f\| := \|f\|_{H^2}.$$
Recall the following formula \cite[Sec.~16.9]{Hb2-book} for the norm in an $\HH(b)$ space:
\begin{equation}\label{eq:norm in H(b)}
\| f \|_{\HH(b)}^2=\sup_{g\in H^2}(\|  f+bg \|^2-\|  g \|^2).
\end{equation}
From Corollary \ref{co:who contains the constants} and the assumption that $\phi(0) = 0$,  we know  that 
$$ C_\phi (\mathcal{H}(b))\subset  \mathcal{H}(b \circ \phi).$$

\begin{theorem}\label{th:compo inner}
	Let $b \in H^{\infty}_{1}$ and $\phi$ be an inner function such that $\phi(0) = 0$. Then we have the following. 
\begin{enumerate}
    \item  $C_\phi$ is an isometry from $\mathcal{H}(b)$ into $\mathcal{H}(b \circ \phi)$.
    \item We have 
    $$C_\phi(\HH(b))=\HH(b\circ\phi)\cap\YY_\phi$$ and the space $\HH(b\circ\phi)$ can be orthogonally decomposed as
		\begin{equation}\label{eq:decomp H(b circ phi)}
			\HH(b\circ\phi)= C_\phi(\HH(b))\oplus_{	\HH(b\circ\phi)} 	\left(\HH(b\circ\phi)\cap\YY_\phi^\perp\right).
		\end{equation}
\end{enumerate}		 
\end{theorem}

\begin{proof}
For an arbitrary $F\in \HH(b\circ\phi)$, decompose it orthogonally in $H^2$ as 
$$F=f\circ \phi+u, \quad f\circ \phi\in \YY_\phi, \quad u\in\YY_\phi^\perp.$$ From \eqref{eq:norm in H(b)} we have
	\[
		\|F\|^2_{\mathcal{H}(b \circ \phi)}=\sup_{H\in H^2}\left( \|F+(b\circ \phi) H\|^2-\|H\|^2\right).
	\]
	Decomposing $H \in H^2$ as 
    $$H=h\circ \phi+v, \quad h\circ \phi\in \YY_\phi, \quad v\in\YY_\phi^\perp,$$ we have, again using \eqref{eq:norm in H(b)},
	\[
	\begin{split}
		\|F\|^2_{\mathcal{H}(b \circ \phi)}&= \sup_{\substack{h \in H^2 \\ v \perp \YY_{\phi}}} \left(\| f\circ\phi+u+(b\circ\phi)(h\circ \phi+v)\|^2-\|h\circ\phi+v\|^2\right)\\
		&=\sup_{\substack{h \in H^2 \\ v \perp \YY_{\phi}}} \left(\| (f+bh)\circ\phi+u+(b\circ\phi)v\|^2-\|h\circ\phi\|^2-\|v\|^2\right).
	\end{split}
	\]
	By Lemma~\ref{le:mistery identity}, $u+(b\circ \phi)v\in \YY_\phi^\perp$, and so the equality above continues as 
\begin{align*}
& = \sup_{\substack{h \in H^2 \\ v \perp \YY_{\phi}}} \left(\| (f+bh)\circ\phi\|^2+\|u+(b\circ\phi)v\|^2-\|h\circ\phi\|^2-\|v\|^2\right)\\
	&=\sup_{\substack{h \in H^2 \\ v \perp \YY_{\phi}}} \left(\left(\| (f+bh)\circ\phi\|^2-\|h\circ\phi\|^2\right) + \left(\| u+(b\circ \phi)v\|^2-\|v\|^2\right)\right).
\end{align*}
	The first parentheses in the previous line depend only on $ h $, while the second depends only on $ v $. Thus,  we actually have
		\begin{align}\label{eq:suprema}
			\|F\|^2_{\mathcal{H}(b \circ \phi)} & =
		\sup_{h\in H^2} \left(\| (f+bh)\circ\phi\|^2-\|h\circ\phi\|^2\right)\\
		& \quad \quad  +\sup_{v\perp\YY_\phi} \left(\| u+(b\circ \phi)v\|^2-\|v\|^2\right).\nonumber
		\end{align}
Note that both suprema on the right hand side of the previous line are bounded above by $\|F\|^2_{\mathcal{H}(b \circ \phi)}$.

	Since $C_\phi$ acts isometrically on $H^2$, the first supremum in~\eqref{eq:suprema} is
	\begin{equation}\label{eq:first supremum}
			\sup_{h\in H^2} \left(\| (f+bh) \|^2-\|h \|^2_2\right)=\|f\|^2_{\HH(b)}.
	\end{equation}
	In particular, we see that if  $F =   f\circ\phi \in \HH(b\circ\phi) $, then $f \in \HH(b)$ ($u = 0$) and
	\[
	\|f\|^2_{\HH(b)}=\|C_\phi f\|^2_{\HH(b\circ\phi)},
	\]
	which proves (a).

 For the second supremum on the right hand side of \eqref{eq:suprema}, we know that 
	\[
	\| u\|_{\HH(b\circ\phi)}^2=\sup_{h \in H^2}\big( \| u+(b\circ\phi)h\|^2-\|h\|^2\big).
	\]
Now write $h=g\circ \phi+w$ and again use the invariance of $\YY_\phi^\perp$ with respect to multiplication by $b\circ\phi$ (Lemma~\ref{le:mistery identity})  to obtain the orthogonal decompositions
	\[
	\begin{split}
			\| u\|_{\HH(b\circ\phi)}^2&=\sup_{g,w} \left(\| (bg)\circ\phi+u+(b\circ \phi)w\|^2-\|w+g\circ\phi\|^2\right)\\
&		=\sup_{g,w} \left(\|(bg)\circ\phi\|^2+\| u+(b\circ\phi)w\|^2-\|w\|^2-\|g\circ\phi\|^2\right)\\
&=\sup_{g,w}\left((\|bg\|^2-\|g\|^2)+\left(\| u+(b\circ\phi)w\|^2-\|w\|^2\right)\right). 
	\end{split}
		\]
	Since $|b| \leq 1$ on $\D$, the expression in the first parentheses of the last equation is nonpositive, and so we obtain
		\[
			\| u\|_{\HH(b\circ\phi)}^2\leq \sup_w \left(\| u+(b\circ\phi)w\|^2-\|w\|^2\right) .
		\]
		In fact, we have 
	\begin{equation}\label{eq:second supremum}
			\| u\|_{\HH(b\circ\phi)}^2=\sup_w \left(\| u+(b\circ\phi)w\|^2-\|w\|^2\right)
	\end{equation}
	since we may choose $g=0$ in the above. Thus,  $ u\in \HH(b\circ\phi) $. Moreover, combining \eqref{eq:suprema}, \eqref{eq:first supremum}, and \eqref{eq:second supremum}, we obtain the equality
	\[
		\|F\|^2_{\mathcal{H}(b \circ \phi)}=\| f\circ\phi\|^2_{\HH(b\circ\phi)}+	\| u\|_{\HH(b\circ\phi)}^2.
	\]
	It follows that the orthogonal decomposition $F=f\circ\phi+u$  in $H^2$  is also an orthogonal decomposition in $\HH(b\circ\phi)$. This verifies \eqref{eq:decomp H(b circ phi)} and thus completes the proof. 
	\end{proof}

\section{Final remarks}\label{FRem}

\subsection*{Minimality}
The second part of Theorem \ref{Javad} establishes a minimality condition satisfied by $\eta$ defined in cases (a), (b) or (c) when $b$ and $\phi$ are inner. It would be desirable to obtain similar results for general $\phi$ and $b$.  
The following examples demonstrate the complexities that arise.

	\begin{example}
	    Suppose 
	$$b(z)=z^2  \; \; \mbox{and} \; \; \phi(z)=\tfrac{1}{2} z.$$ We have $\HH(b)=\HH(z^2)$, which is a two dimensional model space $\mathcal{K}_{z^2}$ with a basis formed by $1$ and $z$ \cite[p.~116]{MR3526203} and it is easy to check that  $C_\phi(\HH(b))\subset \HH(b)$. 
	On the other hand, the target space given by Theorem \ref{MainT} (observing that $b(0) = 0$), is $\HH(\eta)$, where
	\[
	\eta(z) :=z\frac{z^2/4}{z/2}=\tfrac{1}{2} z^2.
	\]	However, since $\sup_{z \in \D} |\eta(z)|<1$, we have $\HH(\eta)=H^2$ (with an equivalent norm), which is not a subspace of the two dimensional space $ \HH(b)=\HH(z^2)$.
	\end{example}

    When $b$ is a rational function, the paper \cite{MR3503356} contains  a precise description of $\HH(b)$. Without stating the general case, we will use its consequence several times in the next example.


	\begin{example}
	Let
    $$ \phi(z)=\tfrac{1}{2} (1 - z) \; \; \mbox{and} \; \;  b(z)= \tfrac{1}{2} (1 + z).$$ Since $ \phi $ is nonextreme, $ \HH(\phi) $ contains the constant functions (Lemma \ref{Hbcopnstants}). Thus, the target space given by Theorem~\ref{MainT} is $\HH(\eta)$, where
		\[
		\eta(z):= b(\phi(z))= \tfrac{1}{4}(3-z).
		\]
By \cite[Thm.~1.6]{MR3503356}, we obtain 
    $$\HH(\eta)=(z+1)H^2+\C.$$

On the other hand, the results in \cite{MR4653340} provide various sufficient conditions for $ C_\phi(\HH(b))\subset \HH(b) $. In particular
	\cite[Cor.~4.12]{MR4653340}, this is true for the above functions $ b $ and $ \phi $, and therefore, from  \cite[Thm.~1.6]{MR3503356}, we have $$C_\phi \HH(b)\subset \HH(b)=(z-1)H^2+\C.$$
	The two spaces $ \HH(b) $ and $ \HH(\eta) $ above are not comparable. Indeed, one can see  that all functions in $ \HH(b) $ have nontangential limits at $ 1 $, while all functions in $ \HH(\eta) $ have nontangential limits at $ -1 $, which allows to construct functions that are in one of the spaces but not in the other. Therefore, neither $\HH(b)$ nor $\HH(\eta)$ is minimal. 
		
		In fact, one can obtain a bit more. We know \cite[Thm.~2.4]{FricainOstermann} that  $$ C_\phi(\HH(b))\subset \HH(b)\cap\HH(\eta) = (1 - z^2) H^2 + \operatorname{span}\{1, z\},$$ and it turns out  \cite[Thm.~1.6]{MR3503356} that the last intersection is also a de Branges--Rovnyak space $\HH(\beta)$ with
		\[
		\beta(z)=\frac{2+\sqrt{3}}{4}+\frac{2-\sqrt{3}}{4}z^2.
		\]

Finally, we note that  $C_{\phi}(\HH(b))$ is actually much smaller than $\HH(\beta)$. To see this, observe that $\phi$ maps $\D$ to an internally tangent disk at $-1$ and so all functions in 
$$C_{\phi}(\HH(b)) = (1 + z)\{f \circ \phi: f \in H^2\} + \C$$ are analytic at all points of $\T \setminus \{-1\}$ and admit nontangential limits at $-1$. However, not all functions in
$\HH(\beta) = (1 - z^2) H^2 + \operatorname{span}\{1, z\}$ are subject to this restriction.

Both examples above show that the spaces obtained through our results concerning $\HH(\eta)$ in the inclusion $C_{\phi} (\HH(b)) \subset \HH(\eta)$ may be far from optimal. The possibility
of an optimal space, as well as identifying it, remains an open problem. Notice that 
the inclusion $\HH(b_1) \subset \HH(b_2)$ for general $b_1, b_2 \in H^{\infty}_{1}$, has been characterized by Ball and Kriete in \cite{MR868573} (see also \cite[Ch.~27]{Hb2-book}). However, the precise condition is quite technical, making the question of minimality in our discussion very challenging. 
%
\end{example}

\subsection*{Invariance of $\HH(b)$}
In light of our previous results, when does $C_{\phi}$ map $\HH(b)$ to itself? When $\phi$ and $b$ are inner functions, this was explored in \cite[Theorem 4.1]{MR3176147} and expanded to rational $b$ in \cite{MR4653340}. In the most general setting,  it does not seem possible to give necessary conditions as long as the question of the minimality of $\HH(\eta)$ remains unsettled. However, using the following result involving {\em contractive} inclusions from \cite[Thm.~27.12]{Hb2-book}, we can establish some sufficient conditions. 

\begin{lemma}\label{le:FM th 27.12}
	Let $ b_1, b_2  \in H^{\infty}_{1}$. Then $ \HH(b_2) $ is contractively contained in $  \HH(b_1) $ if and only if $ b_1=b_2b $ for some $ b \in H^{\infty}_{1}$.
\end{lemma}

 The next theorem, which follows from Theorem \ref{MainT}, provides some sufficient conditions as mentioned above. 

\begin{theorem}\label{th:invariance}
	Suppose $ b, \phi  \in H^{\infty}_1$. Define $ \eta  \in H^{\infty}_{1}$ as follows:
	\begin{enumerate}
		\item  If $ b(0)=0 $, then $$ \eta(z)=z\frac{b(\phi(z))}{\phi(z)}.$$
		
		\item If $ \HH(\phi) $ contains the constants, then $ \eta(z)=b(\phi(z)) $.
		
		\item If neither (a) nor (b) holds, then $ \eta(z)=zb(\phi(z)) $.
		
	\end{enumerate}
If 
\begin{equation}\label{eq:quasikonig}
 b=\eta \widetilde{b}
\end{equation}
for some $\widetilde{b} \in H^{\infty}_{1}$, then $ C_\phi(\HH(b))\subset \HH(b) $.

\end{theorem}

\begin{example}
The following example is an application of Theorem \ref{th:invariance}. 
Consider $$ \phi(z)=\tfrac{1}{2}(1 + z) \; \;  \mbox{and} \; \;   b(z)=e^{z-1}.$$
Since $b$ is not a rational function, the invariance of $ \HH(b) $ under composition by $ \phi $ cannot be obtained  through the concrete description of $ \HH(b) $ as in \cite[Thm.~1.6]{MR3503356}.
With this choice of $b$ and $\phi$ we have 
	$$
	\eta(z)=b(\phi(z))=e^{\frac{z-1}{2}} \; \; \mbox{and} \; \; 
	b(z)=\eta(z)e^{\frac{z-1}{2}},
	$$
	and so \eqref{eq:quasikonig} is verified with $\widetilde{b}(z)=e^{\frac{z-1}{2}} $. Therefore $ C_\phi(\HH(b))\subset\HH(b) $.

\end{example}	

\subsection*{Statement about the use of AI} Gemini Pro was used in this paper, including generating ideas and performing calculations. The authors assume full responsibility for the content of the work and its final form. 

\bibliographystyle{plain}

\bibliography{references}

\end{document}